\documentclass[12pt]{amsart}

\usepackage{latexsym, amsthm, amscd, graphicx, euscript, float}
\usepackage{amssymb,amsthm,amsmath}

\newtheorem{theorem}{Theorem}[section]
\newtheorem{lemma}[theorem]{Lemma}

\newtheorem{proposition}[theorem]{Proposition}

\theoremstyle{remark}
\newtheorem{definition}[theorem]{Definition}
\newtheorem{remark}[theorem]{Remark}
\newtheorem{example}[theorem]{Example}
\newtheorem{problem}[theorem]{Problem}

\numberwithin{equation}{section}

\newcommand{\bC}{{\mathbb C}}


\begin{document}

\title[Rational functions with nodes] 
{\vspace*{.1cm} Rational functions with nodes}


\author[Masayo Fujimura]{\noindent Masayo Fujimura}
\email{masayo@nda.ac.jp}
\address{Department of Mathematics, National Defense Academy of Japan}

\thanks{The first author was partially supported by 
           Grants-in-Aid for Scientific Research (C) (Grant No. 22540240).}

\author[Masahiko Taniguchi]{\noindent Masahiko Taniguchi}
\email{tanig@cc.nara-wu.ac.jp}
\address{Department of Mathematics, Nara Women's University}
\thanks{The second author was partially supported by 
           Grants-in-Aid for Scientific Research (C) (Grant No. 23540202)
           and Grant-in-Aids for Scientific Research (B) (Grant No. 25287021).}

\begin{abstract} 
  A natural kind of compactification of the virtual 
  moduli spaces of rational functions of one 
  complex variable is given. To describe the boundary points geometrically, 
  the authors introduce the concept of
  rational functions with nodes, defined on partially crushed 
  punctured Riemann spheres with nodes.
\end{abstract}

\maketitle

\section{Introduction and main results}

A {\em dynamical structure} of a rational function 
$R:\hat{\bC}\to \hat{\bC}$ is the 
M\"obius conjugacy class of $R$. 
The {\em moduli space of rational functions} of degree $d$
is the set of all dynamical structures of rational functions of degree $d$,
and is denoted by ${M}_d $. For the details and backgrounds, see for instance,
\cite{FT08}, \cite{M}, \cite{MNTU}, and \cite{S}.

In this note, we introduce 
``rational functions with nodes'' and the dynamical structures of them. 
Here, we also 
consider a natural kind of 
marking for such functions, and hence actually we discuss about  
marked rational functions with nodes.
To give precise definitions of them, 
first we recall that 
a {\em $\partial$-marked 
$n$-punctured Riemann sphere $\hat{S}=({\mathcal  D}(\hat{S}), N(\hat{S}))$  
with nodes} is  
a pair of a family ${\mathcal  D}(\hat{S})= \{D_m\}_{m=1}^M$ of
$\partial$-marked $n_m\, (\geq 3)$-punctured Riemann spheres 
$D_m$ 
and the node set $N(\hat{S})$ consisting of 
$J$ pairs $\{p_j,p'_{j}\}$ of 
punctures of different $D_{m(j)}$ and $D_{m'(j)}$,
which satisfy that
\[
\sum_{m=1}^M\, n_m = 2J + n, \qquad J = M-1 \leq n-3,
\]
and that
the topological space $\hat{S}^*$ 
obtained from the disjoint union 
$|{\mathcal  D}|(\hat{S})=\sum_{m=1}^M\, D_m$ 
by filling and identifying all pairs of punctures in 
$N(\hat{S})$ is connected. Here, equipping $\hat{S}^*$ with a 
$\partial$-marking induced from that of $\hat{S}$,
we call $\hat{S}^*$ the {\em realization} of 
$\hat{S}$. 
Cf., for instance, \cite{B}, \cite{Bu}, \cite{FunT}, and \cite{IT}. 
Also see Definition \ref{mark}
below for the precise definition of $\partial$-marking.

Punctures of $\hat{S}$ in $N(\hat{S})$ are called {\em nodal}, 
and {\em non-nodal} otherwise.
When we ignore $\partial$-marking, we call $\hat{S}$ simply 
an {\em $n$-punctured Riemann sphere with nodes}.
In the sequel, we regard punctured Riemann spheres without nodes also as  
those with (no) nodes.

Now, the limit of rational functions may degenerate to the identity function 
on some components of the 
punctured Riemann sphere with nodes obtained as the limit of 
the punctured Riemann spheres where the functions are defined. 
To describe such phenomena precisely, we generalize  the above definition as follows. 

\begin{definition}[Partial crush]\label{crush}
{\rm 
A {\em partial crush of level $n$} is an ordered pair $(\hat{S}, \hat{R})$ of 
an $n$-punctured Riemann sphere 
$\hat{S}=({\mathcal  D}(\hat{S}), N(\hat{S}))$  
with nodes and a pair $\hat{R}$ 
of a family ${\mathcal  D}(\hat{R})= \{D_i'\}_{i=1}^I$ of
$n_i$-punctured Riemann spheres 
$D_i'$ 
and the node set $N(\hat{R})$ consisting of 
the pairs $\{p_r,p'_{r}\}$ of 
punctures of different $D_{i(r)}'$ and $D_{i'(r)}'$ 
such that
\[
{\mathcal  D}(\hat{R}) \subset {\mathcal  D}(\hat{S}), \qquad N(\hat{R})
\subset N(\hat{S}),
\]
which satisfies two additional conditions:
Every pair in $ N(\hat{S})$ contains a puncture of $\hat{R}$
and  every component in ${\mathcal  D}(\hat{S})- {\mathcal  D}(\hat{R})$ 
has at least two non-nodal punctures of $\hat{S}$. 

We say that two partial crushes  $(\hat{S}_1, \hat{R}_1)$ 
and $(\hat{S}_2, \hat{R}_2)$ of level $n$ are {\em equivalent} if 
there is a homeomorphism $f: |{\mathcal  D}|(\hat{S}_1)
\to |{\mathcal  D}|(\hat{S}_2)$ which 
preserves punctures and nodes and is a conformal map of 
 $|{\mathcal  D}|(\hat{R}_1)$ onto 
$|{\mathcal  D}|(\hat{R}_2)$. 
A {\em partially crushed 
$n$-punctured Riemann sphere 
with nodes} is the equivalence class $[\hat{S}, \hat{R}]$ of 
a partial crush $(\hat{S}, \hat{R})$ of level $n$.
}
\end{definition}

In the sequel, 
we abbreviate $[\hat{S}, \hat{R}]$ to $\hat{R}$ when it causes no confusion.
Also, we call a component in ${\mathcal  D}(\hat{S})- {\mathcal  D}(\hat{R})$
a {\em crushed component} of $\hat{R}$,
though it is not a component in ${\mathcal  D}(\hat{R})$. 

\begin{remark}
{\rm
Let $\overline{D_i'}^*$ be the closure of 
$D_i'\in  {\mathcal  D}(\hat{R})$ in $\hat{S}^*$. Then by definition, 
every component of $W=\hat{S}^* - \bigcup_{i=1}^I\, \overline{D_i'}^*$ 
is a punctured Riemann sphere belonging to 
${\mathcal  D}(\hat{S})- {\mathcal  D}(\hat{R})$. 
The second additional condition 
that every crushed component has at least 
two non-nodal punctures is crucial  
to give a natural kind of definition of ``rational function with node". 
Also see the proof of Theorem \ref{thm:1}.
}
\end{remark}

Such an $[\hat{S}, \hat{R}]$ can be characterized 
by the crush data as follows.
For every crushed component $D$, let $L(D)$ and 
$B(D)$ be the number of all non-nodal punctures 
of $\hat{S}$ on $D$ and the set of all  
punctures of $\hat{R}$ in $N(\hat{S})$ 
paired with punctures on $D$, respectively. 
We call $B(D)$ a {\em singular bouquet} of punctures of $\hat{R}$, 
and $L(D)$ the {\em level} of the singular bouquet $B(D)$, 
which is also denoted by $L(B(D))$. 
Here note that every singular bouquet contains 
at most one puncture of $D_i'$ 
for every $D_i'$ in ${\mathcal  D}(\hat{R})$. 
A puncture $p$ in a singular bouquet $B$ is called {\em singular}, 
and we call $L(B)$ also the {\em level} $L(p)$ 
of $p$. 

Let $\{B_1, \cdots, B_N\}$  $(N=M-I)$ be 
the maximal set of the singular bouquets of $\hat{R}$. Then  
the set of all pairs $\{(B_\ell, L(B_\ell))\mid \ell=1, \cdots, N\}$ is called  
the {\em crush data of $\hat{R}$}.
Let $\{q_1, \cdots, q_A\}$ be the set of 
all non-nodal non-singular punctures of $\hat{R}$, and we have 
\[
A + \sum_{\ell =1}^N\, L(B_\ell) = n.
\]
We set 
$X(\hat{R})=\{q_1, \cdots, q_A, B_1, \cdots, B_N\}$ and write $X(\hat{R})$ 
also as $\{q_r\}_{r=1}^{A+N}$. Further, set $L(q_r) = 1$ for every $r\leq A$.

\begin{remark}
{\rm
The crush data of $\hat{R}$ depend on the choice of $\hat{S}$.
But, possible choice of 
the set of the singular bouquets and their levels is finite in number. 
Also by the second additional condition in Definition \ref{crush}, 
the level of every singular bouquet, or of every singular puncture, 
is not less than $2$. 
}
\end{remark}

\begin{definition}[$\partial$-marking]\label{mark}
{\rm 
A {\em $\partial$-marking}
of a partially crushed 
$n$-punctured Riemann sphere $[\hat{S},\hat{R}]$
with nodes is a surjection $\iota$ of $\{1,\cdots, n\}$ to the set
$X(\hat{R})$ such that $\iota^{-1}(q_r)$ consist  
of $L(q_r)$ values for every $q_r\in X(\hat{R})$.
If there are no crushed components, 
a $\partial$-marking is just an order of all non-nodal punctures.

We say that  a partially crushed 
$n$-punctured Riemann sphere  $[\hat{S},\hat{R}]$, or simply  $\hat{R}$, 
with nodes is {\em marked} if we equip $\hat{S}$ with a $\partial$-marking,  
which also canonically induces the $\partial$-marking of $\hat{R}$. 
In the sequel,
a marked $\hat{R}$ is denoted by the same $\hat{R}$  
unless we need the $\partial$-marking explicitly. 
}
\end{definition}

A $\partial$-marking of $\hat{R}$ can be considered 
 also as  
an ordered set of 
disjoint subsets $E_r$ consisting of $L(q_r)$ values in 
$\{1,\cdots, n\}$ for every $q_r\in X(\hat{R})$ 
which cover $\{1,\cdots, n\}$, and is 
always induced  canonically  from the $\partial$-marking 
of $\hat{S}$, where $\hat{R}=[\hat{S}, \hat{R}]$. 

\begin{definition}[Realization]
{\rm 
For a marked partially crushed $n$-punctured Riemann sphere $\hat{R}$
with nodes, let $\hat{R}^*$ be 
obtained from $|{\mathcal  D}|(\hat{R})$ 
by filling and identifying every pair of punctures in 
$N(\hat{R})$ by one point, which is called a {\em non-singular node} of 
$\hat{R}^*$. 
Equipping  $\hat{R}^*$ with the $\partial$-marking of $\hat{R}$,
we call $\hat{R}^*$ the {\em realization} of 
$\hat{R}$. 
}
\end{definition}

\begin{remark}
{\rm 
$\hat{R}^*$ is not necessarily connected. 
If it is connected, then every singular bouquet 
consists of a single singular puncture.
}
\end{remark}

We call every connected component of $\hat{R}^*$ 
an {\em ordinary part} of $\hat{R}$, while 
every $D_i$ of ${\mathcal  D}(\hat{R})$ an {\em ordinary component}  
of $\hat{R}$. 

\begin{definition}[Rational function with nodes]\label{RFN}
{\rm
A {\em rational function $({\mathcal  F}, \hat{R})$ with nodes of type $d$} is 
a family ${\mathcal  F}=\{F_i\}_{i=1}^I$ of 
rational functions $F_i$ on ordinary components $D_i'$ 
of a partially crushed $(d+1)$-punctured Riemann sphere 
$\hat{R}=(\{D_i'\}_{i=1}^I, N(\hat{R}))$ with nodes 
satisfying the following conditions.
\begin{enumerate}
  \item
    Every function $F_i$ is not the identity and has its (not necessarily simple)  
    fixed points only at punctures $p$ of $D_i'$ with   
    multiplicity not greater than the level $L(p)$ of $p$.
  \item
    ({\bf Index formula at nodes}) Every nodal puncture of $\hat{R}$ 
    is either a simple fixed point of, 
    or not fixed by, the function in ${\mathcal  F}$ corresponding to 
    the puncture. 
    The sum of the dynamical indices at the pair of  
    punctures in the same node is $1$.
\end{enumerate}
}
\end{definition}

Next, recall that the {\em generic locus} $GM_d$ of the moduli space 
$M_d$ is the sublocus corresponding to all 
rational functions of degree $d$ 
with simple fixed points only, which are called {\em generic}. 
A marking of a generic rational function $F$ is 
 an order of $d+1$ fixed points of $F$. 
The set of the 
dynamical structures of all marked generic rational functions of degree $d$ 
is called the {\em generic virtual  moduli space} 
of degree $d$, and is denoted by $GVM_d$. 
Then as in \cite{M}, there are two canonical projections of $GVM_d$: 
Using the notations as in \cite{FunT} and \cite{T}, 
the {\em Milnor projection}
$\rho : GVM_d \to V {\rm Conf}(d+1, \widehat{\bC})$
maps every point $[F]$ of $GVM_d$ to the 
M\"obius equivalence class of the ordered set of $d+1$ fixed points of $F$, and 
the {\em index decoration}
${\Lambda}: GVM_d \to (\bC^*)^{d+1}$
maps $[F]$ to a $(d+1)$-dimensional vector 
\[
\Lambda([F]) = (\lambda_1([F]), \cdots, \lambda_{d+1}([F])),
\] 
where $\lambda_r([F])$ is  
the dynamical index of $F$ at
 the $r$-th fixed point $p_r$ 
for every $r$. 
The pair $(\rho, {\Lambda})$ of these projections 
 gives a biholomorphic injection of $GVM_d$.

Now, we introduce a similar kind 
of marking and index decoration as above 
for rational functions  with nodes.

\begin{definition}[Reduced index decoration]
{\rm
A {\em marking} of a rational function $({\mathcal  F}, \hat{R})$ with nodes 
is the 
$\partial$-marking $\iota$ of $\hat{R}$,  
or equivalently, the ordered set $\{E_r\}_{r=1}^{A+N}$. 
A marked rational function with nodes is denoted by a triple 
$({\mathcal  F}, \hat{R}, \iota)$, or more precisely, by 
$({\mathcal  F}, \hat{R}, \{E_r\})$.

Next, the {\em reduced index decoration} $\Lambda^\#$ for 
$({\mathcal  F}, \hat{R}, \iota)$ is defined at every puncture of 
$\hat{R}$ as follows: 
If $q\in X(\hat{R})$ is a non-singular puncture of $\hat{R}$, then 
the value of $\Lambda^\#$ at $q$ is the index of 
 $F_q$ at $q$, where $F_q$ is the function in ${\mathcal  F}$ 
corresponding to $q$. 
If $q\in X(\hat{R})$ is a singular bouquet $B$, then 
for every singular puncture $p\in B$, 
the value of $\Lambda^\#$ at $p$ is the coefficients 
$(c_{1}, \cdots c_{L})\in \bC^{L}$ of the principal part 
\[
\frac{c_1}{z} + \cdots +\frac{c_L}{z^L} 
\]
of the Laurent series expansion 
of $1/(z-F_p(z))$ at $p=0$ with the global coordinate $z$ 
of $\bC$ and $L=L(B)$. 
Here, using the cyclic order of punctures induced from the $\partial$-marking 
on the component corresponding to $p$, we take two ordered  punctures ``adjacent" to $p$ 
on the component, and send $p$ and these two punctures to $0,1,\infty$, respectively.

Attaching the vectors defined as above 
to all singular punctures, we obtain a point of 
$\bC^{A + \sum_\ell\, \#(B_\ell)L(B_\ell)}$ 
as the reduced index decoration $\Lambda^{\#}$ 
for  $({\mathcal  F}, \hat{R}, \iota)$,
where $\#(B)$ is the cardinality of $B$.
}
\end{definition}

\begin{remark}
{\rm 
In particular, $c_{1}$ in the above definition 
is the index of $F_p$ at $p=0$. 
(Cf. \cite{FT10} and \cite{T}.) Here recall that 
the indices satisfy the {\em index formula}, i.e., the 
sum of all indices on an ordinary component is $1$. 
Recall that $c_\ell$ except for $c_1$ depend on the choice of punctures 
normalized to be $1$ and $\infty$.
}
\end{remark}

\begin{definition}[Dynamical structure]
{\rm
We say that two marked 
rational functions 
$({\mathcal  F}, \hat{R}, \iota)$ and 
$({\mathcal  G}, \hat{R}', \iota')$ with nodes of type $d$ 
are {\em marking-preserving M\"obius conjugate} if
\begin{enumerate}
  \item
    $[\hat{R}]=[\hat{R}']$ including the $\partial$-markings, i.e., 
    there are M\"obius transformations $T_i$ for all $i$ such that 
    $D_i'=T_i(D_i)$ which preserve the 
    $\partial$-markings and nodes including the order, 
    where  
    ${\mathcal  D}(\hat{R})= \{D_i\}_{i=1}^I$ and  
    ${\mathcal  D}(\hat{R}')=\{D_i'\}_{i=1}^I$, and
  \item
    ${\mathcal  F}=\{F_i\}_{i=1}^I$ is {\em marking-preserving 
    M\"obius conjugate to} ${\mathcal  G}=\{{G_i}\}_{i=1}^I$ by the above 
    $\{T_i\}_{i=1}^I$, 
    i.e., $G_i= T_i\circ F_i\circ T_i^{-1} $
    for every $i$.
\end{enumerate}
A {\em dynamical structure} of a marked 
rational function $({\mathcal  F}, \hat{R}, \iota)$ with nodes 
is the marking-preserving 
M\"obius conjugacy class of it, which
 is denoted by 
$[{\mathcal  F}, \hat{R}, \iota]$. Note that the reduced index decoration 
$\Lambda^{\#}$ is well-defined for the class $[{\mathcal  F}, \hat{R}, \iota]$.

The {\em virtual moduli space of 
rational functions  with nodes of type $d$}
is the set of the dynamical structures of all marked 
rational functions with nodes 
of type $d$,
and is denoted by $\overline{VM}_d $. 
The sublocus of $\overline{VM}_d $ 
corresponding to 
usual rational functions of degree $d$ is called 
the {\em virtual moduli space of degree $d$} and denoted by $VM_d$.
Here, we regard that $\overline{VM}_d $ contains the ``empty" point, 
i.e. the point $\infty=[\emptyset,\emptyset, \emptyset]$.
}
\end{definition}

Finally, we introduce a standard kind of topology on $\overline{VM}_d $. 
For this purpose, we use the {\em reduced realization} 
$\hat{R}^{\#}$ of $\hat{R}$, which is obtained 
from $\hat{R}^{*}$ 
by filling all singular punctures in the same singular bouquet $B$ 
by a single point $q$ for every $B$. Every attached point $q$ 
is called a {\em 
singular node}, which is distinguished 
from other ordinary points, 
even if $B$ consists of a single singular puncture. 
We attach to $\hat{R}^{\#}$ the marking 
induced from that of $\hat{R}$, by replacing singular bouquets to 
singular nodes. 
Note that, even if $\hat{R}^{*}$ is disconnected, 
$\hat{R}^{\#}$ is always connected.

\begin{definition}[Carath\'eodory convergence]
{\rm
We say that points 
$[{\mathcal  F}_k, \hat{R}_k, \iota_k]$ 
converge  to $[{\mathcal  F}, \hat{R}, \iota]$ in $\overline{VM}_d $ 
{\em in the sense of Carath\'eodory} as $k\to \infty$ 
if there is an
admissible sequence of 
continuous surjections 
$f_k:\hat{R}_{k}^{\#} \to \hat{R}^{\#}$ 
such that
$F_{i,k}\circ f_k^{-1}$ converge to $F_i$ spherically 
uniformly on $D_i-U$ for every ordinary 
component $D_i$ of $\hat{R}$ and 
every neighborhood $U$ of $N(\hat{R})\bigcup S(\hat{R})$, 
where
$F_{i,k}$ is a suitable rational function marking-preserving 
M\"obius conjugate to the element in ${\mathcal  F}_k$  
defined on the ordinary component $D_{i,k}$ of $\hat{R}_k$ 
containing $f_k^{-1}(D_i)$, and $S(\hat{R})$ is the set of 
all singular nodes of $\hat{R}$.

Here,  
we say that a sequence 
$\{f_k:\hat{R}_{k}^{\#} \to \hat{R}^{\#} \}$ 
is {\em admissible} if 
\begin{enumerate}
  \item
    $f_k^{-1}$ is a homeomorphism of $D_i$ into an ordinary component 
    of $\hat{R}_k$ 
    for every $k$ and every ordinary component $D_i$ of $\hat{R}$,
  \item
    $f_{k}^{-1}(p)$ is either a non-singular 
    node of $\hat{R}_{k}^*$ 
    or a simple closed curve on an ordinary component of $\hat{R}_k$
    for every $k$ and every non-singular node $p$ of $\hat{R}^*$, 
  \item
    the relative boundary of $f_k^{-1}(\hat{R}^*)$ in $\hat{R}_k^*$ 
    consists of a 
    finite number of non-singular nodes and simple closed curves on 
    ordinary components, and 
    $f_{k}^{-1}(p)$ is a connected component of 
    $\hat{R}_k^{\#}-f_k^{-1}(\hat{R}^*)$ 
    for every $k$ and every $p \in S(\hat{R})$, 
  \item
    the surjection $f_k^\#:X(\hat{R}_k)\to X(\hat{R})$ induced by $f_k$ 
    satisfies that $f_k^\# \circ \iota_k= \iota$ for every $k$, 
    and 
  \item
    for every neighborhood $V$ of the set of all punctures of $\hat{R}$ 
    and every positive $\epsilon$, 
    $f_{k}^{-1}$  is a
    $(1+\epsilon)$-quasiconformal map of $|{\mathcal  D}|(\hat{R}) -V$ 
    for every sufficiently large $k$. 
\end{enumerate}
}
\end{definition}

\begin{definition}[Strong convergence]
{\rm 
We say that 
$[{\mathcal  F}_k, \hat{R}_k, \iota_k]$ 
converge {\em strongly}  to $[{\mathcal  F}, \hat{R}, \iota]$
in $\overline{VM}_d $  as $k\to \infty$ 
if  $[{\mathcal  F}_k, \hat{R}_k, \iota_k]$ 
converge  to $[{\mathcal  F}, \hat{R}, \iota]$ 
 in the sense of Carath\'eodory  as $k\to \infty$  and if 
$[{\mathcal  F}, \hat{R}, \iota]$ is {\em maximal}  
in the sense that, if a subsequence of $[{\mathcal  F}_k, \hat{R}_k, \iota_k]$ 
converges  to another 
$[{\mathcal  G}, \hat{R}', \iota']$ 
in the sense of Carath\'eodory  as $k\to \infty$, then
$[{\mathcal  G}, \hat{R}', \iota']$ is {\em subordinate} to  
$[{\mathcal  F}, \hat{R}, \iota]$,  
i.e., 
we can find a continuous surjection
$\phi: \hat{R}^{\#} \to (\hat{R}')^{\#}$ 
such that
\begin{enumerate}
  \item
    $\phi^{-1}$ is a conformal map 
    of $D_i'$ onto an ordinary component of $\hat{R}$ for every 
    ordinary component $D_i'$ of $\hat{R}'$. 
  \item
    $\phi^{-1}(p)$ is  a non-singular 
    node of $\hat{R}^*$ for  every non-singular node $p$ 
    of $(\hat{R}')^*$,
  \item
    the relative boundary of $\phi^{-1}((\hat{R}')^*)$ in $\hat{R}^*$ 
    consists of a 
    finite number of non-singular nodes and 
    $\phi^{-1}(p)$ is a connected component of 
    $\hat{R}^{\#}-\phi^{-1}((\hat{R}')^*)$ 
    for every $p \in S(\hat{R}')$, and
  \item
    the surjection $\phi^\#:X(\hat{R})\to X(\hat{R}')$ induced by $\phi$ 
    satisfies that 
    $\phi^\# \circ\iota=\iota'$.
\end{enumerate}
}
\end{definition}

\begin{remark}
{\rm 
Roughly speaking, if $[{\mathcal  F}_k, \hat{R}_k, \iota_k]$ in 
$\overline{VM}_d $ 
converge strongly to $[{\mathcal  F}, \hat{R}, \iota]$
 as $k\to \infty$ if and only if 
${\mathcal  F}_k$ converges  to the identity function  
or not, respectively, exactly 
on the crushed components 
or on the ordinary ones of $\hat{R}$.

In general, there might be ``superfluous" singular nodes 
of such $[{\mathcal  G}, \hat{R}', \iota']$ as above.
}
\end{remark}

By using strong convergence, 
we can introduce a topology on  $\overline{VM}_d $, and 
conclude the following result. 
The proofs of all the assertions stated below will be 
given in the next section.

\begin{theorem}\label{thm:1}
$\overline{VM}_d $ is a compact Hausdorff space. 
\end{theorem}

\begin{definition}[Degree]
{\rm
We call the closure of $GVM_d$ in $\overline{VM}_d $ the 
{\em virtual moduli space of 
rational functions  with nodes of degree $d$}, 
and is denoted by $\widehat{VM}_d $. 
We say that a marked rational function 
$({\mathcal  F}, \hat{R},\iota)$ with nodes is of 
{\em degree} $d$ 
for every $[{\mathcal  F}, \hat{R}, \iota]$ in $\widehat{VM}_d $.
}
\end{definition}

\begin{theorem}\label{thm:2}
The virtual moduli space $\widehat{VM}_d $ of 
rational functions  with nodes of degree $d$ 
is compact, and the natural inclusion map
of $GVM_d$ into $\widehat{VM}_d$ 
can be extended to a continuous injection
$\rho: VM_d \to \widehat{VM}_d$ with dense range.
\end{theorem}

Here, we note the following fact.

\begin{lemma}\label{lem:1}
Every marked rational function $({\mathcal  F}, \hat{R}, \iota)$ 
with nodes of type $d$ such that 
the realization $\hat{R}^*$ is connected is of degree $d$.
\end{lemma}

\begin{example}\label{exm:1}
{\rm 
Suppose that points $[P_k]$ in $VM_d$ represent 
the classes of polynomials $P_k$ of degree $d$, 
and $\rho([P_k])$ converge to 
$[{\mathcal  F}, \hat{R},\iota]$ in $\widehat{VM}_d$. Then,  
every $F_m$ in ${\mathcal  F}$ defined on an ordinary component 
$D_m$ of $\hat{R}$ is either 
M\"obius conjugate to a polynomial or to a constant.

In \cite{FT08}, we use ${\mathcal  F}$ only to 
define the boundary point of such a sublocus of $VM_d$. 
Here, we also take the 
binding manner of them into 
account as $\hat{R}$. Hence the boundary of it 
constructed in this paper is 
larger than that in \cite{FT08}, and actually 
a finite branched cover of that.
Also note that the realization of $\hat{R}$ is always connected 
in this case.
}
\end{example}

\begin{proposition}\label{prop:1}
If $d\leq 5$ then  $\overline{VM}_d = \widehat{VM}_d$.
\end{proposition}

Finally, we will give an example (Example \ref{exm:2}) in the next section 
which shows that $\overline{VM}_d - \widehat{VM}_d$ 
is non-empty for every $d\geq 6$. 

\begin{problem}
{\rm
For $d\geq 6$, find explicit conditions 
for rational function with nodes of type $d$ to be of degree $d$.
}
\end{problem}

\section{Proofs}

\begin{proof}[Proof of Theorem \ref{thm:1}]\ 
It is easy to see that $\overline{VM}_d $ is 
Hausdorff and satisfies the 
second countability axiom, and hence it suffices to show  sequential 
compactness of it.
Thus, the next lemma implies the assertion. 
\end{proof}

\begin{lemma}\label{lem:2}
Let 
$[{\mathcal  F}_k, \hat{R}_k, \iota_k]$ be a sequence 
in $\overline{VM}_d $. 
Then we can find a subsequence which 
converges in $\overline{VM}_d$.
\end{lemma}

\begin{proof}
We prove the assertion only for the case that 
all $[{\mathcal  F}_k, \hat{R}_k, \iota_k]$ 
belong to $GVM_d$, 
and hence in particular,  
${\mathcal  F}_k$ and $\hat{R}_k$ consist of 
a single function $F_{k,1}$ and 
a single component $D_{k,1}$, respectively, for every $k$.  
The general cases can be treated similarly. 
Also, taking a subsequence if necessary,
 we may assume that $[\hat{R}_k]$ converge  to the point  
corresponding to a $\partial$-marked 
$(d+1)$-punctured Riemann sphere $\hat{S}=({\mathcal D}(\hat{S}), N(\hat{S}))$  
with nodes in the 
standard compactification 
$\widehat{V{\rm Conf}}(d+1, \widehat{\bC})$ of 
$V {\rm Conf}(d+1, \widehat{\bC})$. 
(Cf., for instance, \cite{B}, \cite{Bu},  and \cite{FunT}.)

Fix a component $D_m$ in ${\mathcal D}(\hat{S})$, and 
let $\{p_r\}$ 
be the set of all punctures of $D_m$. 
Here, we may assume that $\{p_r\}\subset \bC$. 
Furthermore, by taking a suitable representatives 
of $D_{k,1}$ 
and  a subsequence if necessary, 
we may assume 
that $j$-th punctures $p_{k,j}$ of $D_{k,1}$ 
converge to one of $p_r$ for every $j$.
 Let $Y(p_r)$ be the set of all $j$ 
such that $p_{k,j}$ converge to $p_r$ 
and $L'(p_r)$ is the number of such $j$  for every $p_r$. 

Again taking a subsequence if necessary, 
we may assume that, if $L'(p_r)=1$,  then  
the indices $\Lambda_k(j)$ of $F_{k,1}$ at $p_{k,j}$ 
converge to 
a value, say $\Lambda(p_r)$, in $\widehat{\bC}$,
where $j$ is the unique element of $Y(p_r)$, 
 and if $L'(p_r)>1$, then 
\[
\sum_{j\in Y(p_r)}\, \frac{\Lambda_k(j)}{z-p_{k,j}}\quad \mbox{tend to}\quad 
 \sum_{\ell=1}^{L'(p_r)}\, \frac{c_{\ell,r}}{(z-p_r)^\ell}
\] 
with respect to the global coordinate $z$ on $D_m$,
where 
some of $c_{\ell,r}$ might be $\infty$. We define 
 a rational function $F_{p_r}$ by setting 
\[
\frac{1}{z-F_{p_r}(z)}=
 \sum_{\ell=1}^{L'(p_r)}\, \frac{c_{\ell,r}}{(z-p_r)^\ell},
\]
where we regard that $F_{p_r}$ is the identity function $id.$ 
when some of $c_{\ell,r}$ are $\infty$. 

If there exists either $\infty$ among $\Lambda(p_r)$ 
or $id.$ among $F_{p_r}$ on $D_m$, then we classify 
the component $D_m$ as
a crushed one, and if not, as an ordinary one. 
Note that this classification does not depend 
on the representative of $D_m$. Hence we can define
 canonically 
a $\partial$-marked partially crashed $(d+1)$-punctured 
Riemann sphere $[\hat{S},\hat{R}_0]$  with nodes, 
where 
$\hat{R}_0=({\mathcal  D}(\hat{R}_0), N(\hat{R}_0))$ with 
${\mathcal  D}(\hat{R_0})=\{D_i'\}$ consisting of 
all ordinary components in this classification.  
Then, 
every crushed component should contain at least two punctures 
of $\hat{S}$, for if not, no components adjacent to it can be 
ordinary. 
Also, it is easy to construct an 
admissible family of continuous surjections 
$f_k:\hat{R}_k^\#=\hat{R}_k\to \hat{R}_0^{\#}$ such that 
$F_{k,1}\circ f_k^{-1}$ converges to a rational function $F_{D_i'}$ 
defined by
\[
\frac{1}{z-F_{D_i'}(z)}= \sum_{L'(p_r)=1}\, 
\frac{\Lambda(p_r)}{z-p_r} + 
\sum_{L'(p_r)>1}\, \frac{1}{z-F_{p_r}(z)}
\]
 locally uniformly on every $D_i'$, where $\{p_r\}$ is 
as above with $D_m=D'_i$. 

Finally, let ${\mathcal  F}$ be the set of all these $F_{D_i'}$ 
on ordinary components $D_i'$ of $\hat{R}_0$. 
Then it is easy to see that 
$[{\mathcal  F}_k, \hat{R}_k, \iota_k]$ converge strongly to  
$[{\mathcal  F}, \hat{R}_0, \iota]$ 
with naturally induced marking $\iota$, 
and we have the assertion.
\end{proof}

Next, we note the following fact.

\begin{lemma}\label{lem:3}
The natural inclusion map of 
$GVM_d$ into $\widehat{VM}_d$ 
can be extended canonically 
to a continuous injection $\rho :VM_d \to \widehat{VM}_d$. 
\end{lemma}

\begin{proof} 
Every point $[F]$ in $VM_d-GVM_d$  corresponds such an 
$\hat{R}=({\mathcal D}(\hat{R}), \emptyset)$ that ${\mathcal D}(\hat{R})$
consists of a single marked punctured 
Riemann sphere, say $D$,   
and the $\partial$-marking $\iota$ is determined from $F$. 
Here, multiple fixed points of $F$ correspond to 
singular punctures of $\hat{R}$ and every singular bouquet 
consists of a single singular puncture. The 
level of every singular puncture $p$ equals 
the multiplicity of the fixed point of $F$ at $p$.

Now, we set $\rho([F])$ to be the point 
$[\{F\}, \hat{R}, \iota]$ of $\overline{VM}_d$. Then 
from the construction, the realization $\hat{R}^*$ is 
 connected. Hence Lemma \ref{lem:1}, which is proved next but independently, 
implies that $\rho([F])$ is contained in $\widehat{VM}_d$.
>From the definition of topology, 
it is easy to conclude that
the map
\[
\rho: VM_d \to \widehat{VM}_d 
\]
defined above is a continuous injection.
\end{proof}

\begin{proof}[Proof of Theorem \ref{thm:2}] 
Denseness of $\rho(VM_d)$ in $\widehat{VM}_d$ 
is trivial from the definition of 
$\widehat{VM}_d$, and we conclude Theorem \ref{thm:2} by 
Theorem \ref{thm:1} and Lemma \ref{lem:3}.\\
\end{proof}

\begin{proof}[Proof of Lemma \ref{lem:1}]
Suppose that a marked rational function 
 $({\mathcal  F}, \hat{R}, \iota)$ with nodes of type $d$ 
admits a connected realization $\hat{R}^*$. 
Then, every singular bouquet $B_\ell$ consists of 
a single singular puncture, say $q_\ell$. 

Take a representative of the component containing $q_\ell$ 
such that $q_\ell=0$ and let the reduced index decoration of 
$F\in {\mathcal  F}$ 
corresponding to $q_\ell$ at $0$ be $(c_1, \cdots c_{L_\ell})$,
where $L_\ell$ is the level of $q_\ell$.

For every $\epsilon =\{\epsilon_\nu\}$ with 
mutually disjoint $\epsilon_\nu$ sufficiently near $0$, 
let $F_{q_\ell, \epsilon}=F_{q_\ell,\{\epsilon_\nu\}}(z) $ be defined by 
\[
 \frac{1}{z- F_{q_\ell, \epsilon}(z)}=\sum_{\nu=1}^{L_\ell}\frac{\lambda_\nu}{z-\epsilon_\nu} 
\left(=\frac{A_{L_\ell-1}z^{L_\ell-1}+\cdots+A_0}{\prod(z-\epsilon_\nu)}\right),
\]
where $\lambda_\nu$ are non-zero and  $A_\nu$ depend on $\lambda_\nu$ 
and $\epsilon_\nu$. 
More precisely,  $F_{q_\ell, \epsilon}(z)$ can be written as  
\begin{align*}
  \frac{1}{z-F_{q_\ell, \epsilon}(z)} 
      &=\frac{ \sum_{\nu=1}^{L_\ell}\lambda_\nu (z-\epsilon_1)\cdots
                  \stackrel{\rotatebox{180}{$\widehat{}$}}{(z-\epsilon_\nu )}
                             \cdots(z-\epsilon_{L_\ell})}
                            {\prod_{\nu=1}^{L_\ell}(z-\epsilon_\nu)} \\
      & = \frac{\sum_{\nu=1}^{L_\ell}\lambda_\nu (z^{L_\ell-1}-\sigma_{\nu}^{(1)}z^{L_\ell-2}
                             +\sigma_{\nu}^{(2)}z^{L_\ell-3}-\cdots
                             +(-1)^{L_\ell-1}\sigma_{\nu}^{(L_\ell-1)})}
                            {\prod_{\nu=1}^{L_\ell}(z-\epsilon_\nu)},
\end{align*}
with the $j$-th elementary symmetric functions  $ \sigma_{\nu}^{(j)} $ of 
$ \epsilon_1,\cdots,\stackrel{\rotatebox{180}{$\widehat{}$}}{\epsilon_\nu },
    \cdots,\epsilon_{L_\ell}$. 
Here, 
$ \stackrel{\rotatebox{180}{$\widehat{}$}}{(z-\epsilon_\nu )}$ and 
$\stackrel{\rotatebox{180}{$\widehat{}$}}{\epsilon_\nu }$ mean 
the deletion of $(z-\epsilon_\nu )$ and 
$\epsilon_\nu $, respectively.
Hence $ A_\nu $ can be expressed as
\[
      \left(\begin{array}{c}
           A_{L_\ell-1} \\ -A_{L_\ell-2} \\ \vdots \\ (-1)^{L_\ell-1}A_0
      \end{array}\right)=
        \left(\begin{array}{cccc}
                 1 & 1 & \cdots & 1 \\
                 \sigma_1^{(1)} &
                 \sigma_2^{(1)} &\cdots &
                 \sigma_{L_\ell}^{(1)}\\
                 \vdots  & &  & \vdots \\
                 \sigma_1^{({L_\ell}-1)}&
                 \sigma_2^{({L_\ell}-1)} & \cdots &
                 \sigma_{L_\ell}^{({L_\ell}-1)}
            \end{array}\right)
        \left(\begin{array}{c}
                 \lambda_1 \\ \lambda_2 \\ \vdots \\ \lambda_{L_\ell}
            \end{array}\right).
\]

Let $M$ be  the ${L_\ell}\times {L_\ell}$ matrix in the right hand side. Then  
we can show that 
\[
 \det M=  \prod_{\mu<\nu}(\epsilon_\mu -\epsilon_\nu). 
\]
Hence, for mutually distinct $ \{\epsilon_\nu  \}$,  also
$ \{A_\nu \} $ determines $ \{\lambda_\nu \} $ uniquely.
Actually, 
the inverse matrix of $M$ is
\[
\left( \dfrac{(-1)^{j-1}\epsilon_k ^{{L_\ell}-j}}
                               {\Delta_k } \right)
\qquad 
\mbox{with}\quad \Delta_k=\prod_{\nu\neq k}(\epsilon_k-\epsilon_\nu).
\]

Now, take mutually distinct $\{\epsilon_\nu \}$ arbitrarily near $0$ 
and 
fix suitable values $\{(-1)^{\nu -1}A_{L_\ell-\nu}\}$ arbitrarily near $\{c_\nu \}$ so that 
the solution $\{\lambda_\nu \}$ 
determined from them consists of 
non-zero values only. 
Let $F_{q_\ell,\epsilon}$ be defined as above 
with these $\epsilon_\nu $ and $\lambda_\nu $ for every $q_\ell$. 

For every component $D$ of $\hat{R}$, let $\{p_{D,r}\}$ and $\{q_{D,\ell}\}$ be 
the sets of all non-singular punctures and of all singular ones, 
respectively, of $D$. 
Then after taking suitable conjugates of $F_{q_{D,\ell},\epsilon}$ if necessary, 
we have a generic rational function $F_{D,\epsilon}$ on $D$ defined by
\[
\frac{1}{z-F_{D,\epsilon} (z)}= \sum_{\{p_{D,r}\}}\, 
  \frac{\lambda'_{p_{D,r}}}{z-p_{D,r}}\, 
+ \sum_{\{q_{D,\ell}\}}\, \frac{1}{z-F_{q_{D,\ell},\epsilon}(z)}.
\]
Here $\lambda'_{p_{D,r}}$ are non-zero values arbitrarily near to the indices 
$\lambda_{p_{D,r}}$ of the corresponding element of ${\mathcal  F}$ at $p_{D,r}$, 
which satisfy the index relation:
\[
\sum_{\{p_{D,r}\}}\, \lambda'_{p_{D,r}} + \sum_{\{q_{D,\ell}\}}\, A_{q_{D,\ell},L_{\ell-1}}=1,
\] 
where $A_{q_{D,\ell},L_\ell-1}$ is $A_{L_\ell-1}$ for $q_{D,\ell}$ as above 
for every $q_{D,\ell}$. Set ${\mathcal  F}_\epsilon= \{F_{D,\epsilon}\}$, where 
$D$ moves all components of $\hat{R}$, and we obtain 
points $[{\mathcal  F}_\epsilon, \hat{R},\iota]$ arbitrarily near to the given 
 $[{\mathcal  F}, \hat{R}, \iota]$ in $\overline{VM}_d$.
 
Now by a standard surgery by reopening non-singular nodes, 
we can approximate such $[{\mathcal  F}_\epsilon, \hat{R},\iota]$
arbitrarily by points in $GVM_d$, which implies that 
$({\mathcal  F}, \hat{R}, \iota)$ is of degree $d$, and 
we have proved Lemma \ref{lem:1}.
\end{proof}

\begin{proof}[Proof of Proposition \ref{prop:1}] 
When $d\leq 4$, it is easy to see that 
 there are no partially crushed $(d+1)$-punctured 
Riemann sphere $\hat{R}$ with nodes such that 
the realization of $\hat{R}$ is disconnected. Thus 
Proposition \ref{prop:1} follows by Lemma \ref{lem:1}. 

Suppose that $d=5$, and let $\hat{R}$ be 
a partially crushed $6$ punctured 
Riemann sphere with nodes such that the realization 
of $\hat{R}$ is disconnected. 
Then there is essentially only one possibility for $\hat{R}$, 
namely, $\hat{R}$ is
 obtained from $6$ punctured 
Riemann sphere $\hat{S}=(\{D_1,D_2,D_3\}, N(\hat{S}))$ with two nodes, which  
connect $D_3$ with $D_1$ and with $D_2$,  
by crushing $D_3$. Hence, 
the single singular bouquet $B$ of $\hat{R}$ consists of 
two punctures of $D_1$ and $D_2$, and the levels of them are $2$. 
We may assume that $D_1=D_2 = \bC-\{0,1\}$, and $0$ corresponds to the 
singular punctures.

Let $[{\mathcal  F}, \hat{R}, \iota]$ be any point in $\overline{VM}_5$, and 
$(\lambda_1^j, \lambda_2^j; (c_1^j,c_2^j))$ be the reduced 
index decoration on $D_j$ corresponding to ${\mathcal  F}$ for each $j$. 
Then $c_1^j$ is determined from $\lambda_1^j$ and $\lambda_2^j$ 
by the index relation on $D_j$.
For every $k$, we consider a generic rational function $F_k$ of degree 
$5$ defined by 
\[
\frac{1}{z-F_k(z)} = \frac{\tilde{\lambda}_{k,1}^1}{z+2}+
\frac{\tilde{\lambda}_{k,2}^1}{z+2-\epsilon_{k,1}}+
 \frac{\tilde{\lambda}_{k,1}^2}{z-2}+
 \frac{\tilde{\lambda}_{k,2}^2}{z-2-\epsilon_{k,2}}+
 \frac{\tilde{c}_{k,1}}{z}+\frac{\tilde{c}_{k,2}}{z-1},
\]
where the fixed points $0, (-1)^j\,2,(-1)^j\, 2 + \epsilon_{k,j}$ correspond to 
punctures $0,\infty,1$ of $D_j$ for each $j$ and every $k$, 
$\tilde{\lambda}_{k,\nu}^j$ and 
$\tilde{c}_{k,\ell}$ are non-zero and 
satisfy that
\[
\sum_{j,\nu}\, \tilde{\lambda}_{k,\nu}^j + \sum_\ell \tilde{c}_{k,\ell} =1,
\]
for every $k$, 
and $\epsilon_{k,j}$ 
converge to $0$ for each $j$ and $\tilde{\lambda}_{k,\nu}^j$ tend to 
$\lambda_{\nu}^j$ for every $j$ and $\nu$ as $k\to \infty$. 
Since $D_3$ is crushed, we should choose $\tilde{c}_{k,\ell}$ so that 
they tend 
to $\infty$.

Now, we set 
\[
\epsilon_{k,j} = \frac{a_k^j}{k^2}
\]
with bounded non-zero $a_k^j$ 
for every $k$ and $j$, which are determined below.
Take the 
conformal embeddings of $\bC-\{0,1, \pm 2, (-1)^j2+\epsilon_{k,j}\}$ 
into $\bC-\{0,1\}$ which fix $0$ and send $(-1)^j2$ and 
$(-1)^j2+\epsilon_{k,j}$ to $\infty$ and $1$, respectively, 
and hence are the M\"obius transformations
\[
S_{k,j}(z) = \mu_{k,j} \, \frac{z}{z-(-1)^j2}
\qquad \mbox{with}\quad 
\mu_{k,j} = \frac{\epsilon_{k,j}}{(-1)^j2+\epsilon_{k,j}}\,
\approx 
\frac{\epsilon_{k,j}}{(-1)^j\,2}.
\]
(Here and in the sequel, $a_k \approx b_k$ means 
$\lim_{k\to \infty}\, a_k/b_k =1$.)
Note that 
\[
\tilde{c}_{k,1}+ \tilde{c}_{k,2} \approx  c_1^1+c_1^2-1
=1 -\sum_{j,\nu}\, \lambda_\nu ^j,\quad 
\mbox{and}\quad 
\tilde{c}_{k,1}S_{k,j}(1) \approx -c_2^j.
\]

Now, if  $c_2^j$ is non-zero, 
then we set $b_k^j = -c_2^j$, and if not, then we set 
$b_k^j = 1/k$, for every $k$ and each $j$. 
Then we can find bounded non-zero values $a_k^j$ 
which satisfy the equation
\[
\frac{b_{k,1}}{S_{k,1}(1)} = \frac{b_{k,2}}{S_{k,2}(1)}
\]
which we take as $\tilde{c}_{k,1}$ for every $k$. 
Here note that, if $b_k^2 = o(b_k^1)$, for instance, 
then we can take such $a_k^j$ that $a_k^2 = o(a_k^1)$ 
and hence $\epsilon_{k,2}=o(\epsilon_{k,1})$.

These $F_k$ determine the points in 
$GVM_d$ with the marking induced from  above. 
By construction, $\tilde{c}_{k,\ell}$ tend to $\infty$, 
and hence we can see that they converge to 
 $[\hat{F}, \hat{R}, \iota]$  
as $k \to \infty$. 
Thus we conclude the assertion. 
\end{proof}

Finally, we show that Proposition \ref{prop:1} is best possible. Actually, 
$\overline{VM}_d-\widehat{VM}_d$ is non-empty for every $d\geq 6$.

\begin{example}\label{exm:2}
{\rm 
For the sake of simplicity, we consider the case  $d=6$ only, 
for the other cases can be treated by the same arguments.
Let $[{\mathcal  G}, \hat{R},\iota]$ be a point in $\overline{VM}_6$, 
where $\hat{R}=(\{D_1,D_2\}, N(\hat{R}))$ is 
 as in the proof of Proposition \ref{prop:1}, i.e.,   
$\hat{R}^*$ is 
disconnected and  $D_j$ $(j=1,2)$ are  
$\bC-\{0,1\}$, whose punctures at $0$ are 
in the same singular bouquet, which is of level $3$ in this case. 
We set ${\mathcal  G}=\{G_1,G_2\}$, where $G_j$ are 
defined by  
\[
\frac{1}{z-G_j(z)} = \frac{2}{z-1}- \frac{1}{z} 
                    + \frac{(-1)^j}{z^2}+ \frac{1}{z^3}.
\]
Then $[{\mathcal  G}, \hat{R},\iota] \in  \overline{VM}_6 -\widehat{VM}_6$. 
}
\end{example}

Indeed, if not, then it is the limit of a 
suitable sequence of points in $VGM_6$ determined by generic rational 
functions $F_k$ of degree $6$. 
By taking a M\"obius conjugate if necessary, we may assume that such 
$F_k$ are given by
\begin{eqnarray*}
\frac{1}{z-F_k(z)} &=& \frac{\eta_{k,1}}{z+2}+
 \frac{\kappa_{k,1}}{z+2-\epsilon_{k,1}}
+\frac{\eta_{k,2}}{z-2}+ \frac{\kappa_{k,2}}{z-2-\epsilon_{k,2}}\\
& &+ \frac{\lambda_{k,1}}{z} + \frac{\lambda_{k,2}}{z-\delta_k} +
\frac{\lambda_{k,3}}{z-\delta_k'}, 
\end{eqnarray*}
where 
\[
\sum_{j}\, (\eta_{k,j}+\kappa_{k,j}) + \sum_\nu\, \lambda_{k,\nu} = 1,
\]
 $\kappa_{k,j}$ and $\epsilon_{k,j}$ are non-zero 
and tend to $0$, while $\eta_{k,j}$ tend to $2$, 
as $k\to \infty$ for each $j$,  
 and $\lambda_{k,\nu}$ 
tend to $\infty$ as $k\to \infty$ for some $\nu$. 
Also, $\delta_k$ and 
$\delta_k'$  are mutually distinct, equal none of
 $\{0,\pm 2, (-1)^j\, 2+\epsilon_{k,j}\}$, 
and may be assumed to converge to finite values, 
say $a$ and $a'$, respectively. 
Note that $a, a'$ may belong to $\{0,\pm 2\}$.

Now, we may assume that the marking-preserving 
conformal embeddings of 
\[
\bC- \{0,\pm 2, \delta_k,\delta_k', (-1)^j\, 2+\epsilon_{k,j}\}
\]
into $\bC-\{0,1\}$ fix $0$ and send 
$(-1)^j\ 2$ and$ (-1)^j\, 2+\epsilon_{k,j}$ to 
$\infty$ and $1$, respectively, and hence  
are again given by M\"obius transformations
$S_{k,j}(z)$ defined in the proof of Proposition \ref{prop:1}. 

>From the assumption, we may assume 
without loss of generality, that $S_{k,j}(\delta_k)$ and 
$S_{k,j}(\delta_k')$ converge to $0$ for each $j$,  
$ \sum_\nu \lambda_{k,\nu}=-3$ for every $k$, 
and 
$G_{k,j}$ defined by 
\begin{eqnarray*}
\frac{1}{z-G_{k,j}(z)}&=& 
\frac{\lambda_{k,1}}{z} + \frac{\lambda_{k,2}}{z-S_{k,j}(\delta_k)} +
\frac{\lambda_{k,3}}{z-S_{k,j}(\delta_k')}\\
& &+\frac{2}{z-1}
+\frac{2}{z-S_{k,j}((-1)^{3-j}\,2+\epsilon_{k,3-j})}
\end{eqnarray*}
converge to $G_j$ for each $j$ as $k\to \infty$. Set 
\[
\frac{\lambda_{k,1}}{z} + \frac{\lambda_{k,2}}{z-S_{k,j}(\delta_k)} +
\frac{\lambda_{k,3}}{z-S_{k,j}(\delta_k')}= 
\frac{-3z^2 +b_{k,j}z+ a_{k,j}}{z(z-S_{k,j}(\delta_k))(z-S_{k,j}(\delta_k'))},
\]
and write $S_{k,j}(\delta_k)$ and $S_{k,j}(\delta_k')$ simply as $S_{k,j}$ and 
$S_{k,j}'$.  
Then simple computations show that
\begin{align*}
   \lambda_{k,1} &= \frac{a_{k,j}}{ S_{k,j}S_{k,j}'}, \\
   \lambda_{k,2} &= \frac{-a_{k,j}-b_{k,j}S_{k,j}+3S_{k,j}^2}
                        {S_{k,j}( S_{k,j}'-S_{k,j})},\\
   \lambda_{k,3} &= \frac{a_{k,j}+b_{k,j}S_{k,j}'-3(S_{k,j}')^2}
                        {S_{k,j}'( S_{k,j}'-S_{k,j})}.
\end{align*}
Here recall that $(a_{k,j},b_{k,j})$ converge to 
$(1, (-1)^j)$ for each $j$. 

Now, by the first equation, we have
\[
\frac{a_{k,1}}{ S_{k,1}S_{k,1}'}= 
\frac{a_{k,2}}{ S_{k,2}S_{k,2}'},
\]
or more precisely, 
\[
\frac{\mu_{k,1}^2}{a_{k,1}(\delta_k+2)(\delta_k'+2)}
= \frac{\mu_{k,2}^2}{a_{k,2}(\delta_k-2)(\delta_k'-2)}.
\]
Hence
\begin{eqnarray*}
\frac{S_{k,1}( S_{k,1}'-S_{k,1})}{a_{k,1}}
&=&
\frac{\mu_{k,1}^2}{a_{k,1}}\,
\frac{\delta_k}{\delta_k+2}
\frac{2(\delta_k'-\delta_k)}{(\delta_k+2)(\delta_k'+2)}\\
& & \\
&=&
\frac{2-\delta_k}{2+\delta_k}\frac{\mu_{k,2}^2}{a_{k,2}}\,
\frac{\delta_k}{\delta_k-2}
\frac{-2(\delta_k'-\delta_k)}{(\delta_k-2)(\delta_k'-2)}\\
& & \\
&=& 
\frac{2-\delta_k}{2+\delta_k}\frac{S_{k,2}( S_{k,2}'-S_{k,2})}{a_{k,2}}.
\end{eqnarray*}
Here, a rough estimate of the equation for $\lambda_{k,2}$ shows that 
\[
\lambda_{k,2} \approx 
\frac{-a_{k,1}}
{S_{k,1}( S_{k,1}'-S_{k,1})}\approx 
\frac{-a_{k,2}}
{S_{k,2}( S_{k,2}'-S_{k,2})}.
\]
Thus
\[
\frac{2+\delta_k}{2-\delta_k}\approx 1, \quad \mbox{ i.e.,}\quad 
\delta_k \to 0.
\]
Since $a_{k,j}\to 1$ and $\delta_k\not=0$, we conclude that 
\begin{eqnarray*}
& &\frac{-a_{k,1}}
{S_{k,1}( S_{k,1}'-S_{k,1})}- 
\frac{-a_{k,2}}
{S_{k,2}( S_{k,2}'-S_{k,2})} \\
&\approx &
\frac{-2\delta_k}{2-\delta_k}\, \frac{-a_{k,1}}
{S_{k,1}( S_{k,1}'-S_{k,1})},
\end{eqnarray*}
which is non-zero and not 
\[
\frac{o(\delta_k)}
{S_{k,1}( S_{k,1}'-S_{k,1})} 
\]
as $k\to \infty$.

On the other hand, since
\[
S_{k,j} = O(\epsilon_{k,j}\delta_k) = o(\delta_{k}),
\] 
we should have  
\begin{eqnarray*}
0&=&\frac{-a_{k,1}-b_{k,1}S_{k,1}+3S_{k,1}^2}
{S_{k,1}( S_{k,1}'-S_{k,1})}-
\frac{-a_{k,2}-b_{k,2}S_{k,2}+3S_{k,2}^2}
{S_{k,2}( S_{k,2}'-S_{k,2})}\\
&=& \frac{(-a_{k,1}+o(\delta_{k}))}
{S_{k,1}( S_{k,1}'-S_{k,1})}- 
\frac{(-a_{k,2}+o(\delta_{k}))}
{S_{k,j}( S_{k,2}'-S_{k,2})} \\
& =& \frac{-a_{k,1}}
{S_{k,1}( S_{k,1}'-S_{k,1})}- 
\frac{-a_{k,2}}
{S_{k,j}( S_{k,2}'-S_{k,2})} + \frac{o(\delta_k)}
{S_{k,1}( S_{k,1}'-S_{k,1})},
\end{eqnarray*}
which is a contradiction.

\begin{remark}
{\rm 
The condition that the ``last'' components in the 
reduced index decorations at $0$ are non-zero is crucial.

Also note that, even if we consider to approximate 
$[{\mathcal  G}, \hat{R},\iota]$ 
by points in $VM_d-GVM_d$ corresponding to the condition that
 $\delta_k=\delta_k'=0$, 
we still have a contradiction more easily. 
Actually, if the approximating functions 
$H_k$ are defined by 
\begin{eqnarray*}
\frac{1}{z-H_k(z)} &=& \frac{\eta_{k,1}}{z+2}+
 \frac{\kappa_{k,1}}{z+2-\epsilon_{k,1}}
+\frac{\eta_{k,2}}{z-2}+ \frac{\kappa_{k,2}}{z-2-\epsilon_{k,2}}\\
& &+ \frac{\lambda_{k,1}}{z} + \frac{\lambda_{k,2}}{z^2} +
\frac{\lambda_{k,3}}{z^3}, 
\end{eqnarray*}
where $\sum_{j}\, (\eta_{k,j}+\kappa_{k,j}) + \lambda_{k,1} =1$ 
and $\eta_{k,j}$ and $\kappa_{k,j}$ 
tend to $2$ and $0$, respectively, as $k\to \infty$, 
then the corresponding $G_{k,j}$ as before can be defined by  
\begin{eqnarray*}
\frac{1}{z-G_{k,j}(z)}&=& 
-\frac{3}{z} - 
\frac{\epsilon_{k,j}(\lambda_{k,2}+(-1)^j2 \lambda_{k,3})}{4z^2} 
+\frac{\epsilon_{k,j}^2\lambda_{k,3}}{4z^3}\\
& &+\frac{2}{z-1}
+\frac{2}{z-S_{k,j}((-1)^{2-j}\,2+\epsilon_{k,2-j})},
\end{eqnarray*}
which should converge to $G_j$ for each $j$ as $k\to \infty$. 
But this is again impossible.
}
\end{remark}

\vskip 10 mm
\end{document}